\documentclass[11pt]{amsart}
\usepackage[colorlinks=true,linkcolor=blue, anchorcolor=blue, linktocpage=true,urlcolor=blue,citecolor=blue]{hyperref}
\usepackage{amssymb}
\usepackage{graphicx}
\usepackage{mathrsfs}
\usepackage{enumerate}
\usepackage{hyperref}
\usepackage{color}
\newtheorem{theorem}{Theorem}[section]
\newtheorem{definition}{Definition}[section]
\newtheorem{proposition}[theorem]{Proposition}
\newtheorem{lemma}[theorem]{Lemma}
\newtheorem{corollary}[theorem]{Corollary}

\newtheorem{remark}[theorem]{Remark}

\newtheorem{theoremalph}{Theorem}

 \def\NN{{\mathbb N}}

\def \D{\Delta}

\def \cD{\mathcal{D}}

\begin{document}

\title[Physical measures for partially hyperbolic diffeomorphisms]{Physical measures for partially hyperbolic diffeomorphisms with mixed hyperbolicity}

\author[Zeya Mi]{Zeya Mi}

\address{School of Mathematics and Statistics,
Nanjing University of Information Science and Technology, Nanjing 210044, Jiangsu, China}
\email{\href{mailto:mizeya@163.com}{mizeya@163.com}}
%\urladdr{\href{http://math.rice.edu/~hk7/}{http://math.rice.edu/$\sim$hk7/}}

\author[Yongluo Cao]{Yongluo Cao}

\address{Department of Mathematics, Soochow University, Suzhou 215006, Jiangsu, China.}
\address{Center for Dynamical Systems and Differential Equation, Soochow University, Suzhou 215006, Jiangsu, China}
\email{\href{mailto:ylcao@suda.edu.cn}{ylcao@suda.edu.cn}}
%\urladdr{\href{http://math.rice.edu/~zg2/}{http://math.rice.edu/$\sim$zg2/}}

\thanks{Y. Cao is the corresponding author. Z. Mi is partially supported by National Key R\&D Program of China(2022YFA1007800) and NSFC 12271260. Y. Cao is partially  supported by National Key R\&D Program of China(2022YFA1005802) and NSFC 11790274.}

\date{\today}


\begin{abstract}
We study the partially hyperbolic diffeomorphims whose center direction admits the $u$-definite property in the sense that all the central Lyapunov exponents of each ergodic Gibbs $u$-state are either all positive or all negative.
We prove that for this kind of partially hyperbolic diffeomorphisms, there are finitely many physical measures, whose basins cover a full Lebesgue measure subset of the ambient space.
\end{abstract}

\maketitle

%\tableofcontents

\section{Introduction}
%A central topic in smooth dynamical systems is the study of invariant measures which describe the asymptotic distribution of the orbits. The most successful objects in this direction are \emph{physical measures} and \emph{SRB measures}, which were established by Sinai-Ruelle-Bowen in 1970's for uniformly hyperbolic systems\cite{Sin72, Rue76, Bow75, BoR75}. 
Given a smooth diffeomorphism $f$ on a compact Riemannian manifold $M$, one would like to describe asymptotic distribution of the orbits of $f$. From the ergodic viewpoint, such an endeavor has been realized by physical measures.
A \emph{physical measure} of $f$ is an invariant Borel probability $\mu$ on $M$ for which the \emph{basin} of $\mu$
$$
Basin(\mu):=\left\{x\in M: \frac{1}{n}\sum_{i=0}^{n-1}\delta_{f^i(x)}\xrightarrow[n\to +\infty]{weak^*}\mu\right\}
$$
exhibits positive Lebesgue measure. The theory on physical measures were firstly established in uniformly hyperbolic systems \cite{Sin72, Rue76, Bow75, BoR75} by Sinai, Ruelle and Bowen.
It is expected that such theory can be extended for large classes of systems outside uniform hyperbolicity, see \cite[Chapter 11]{BDV05} and \cite{pa,v2,y02} for related conjectures.  
%Moreover, Palis conjectured \cite{pa} that every system can be approximated by one having finitely many physical measures, whose basins cover a full Lebesgue measure subset of the manifold. 

In the present paper, we investigate the existence and finiteness of physical measures in the setting of partially hyperbolic diffeomorphisms (see \S \ref{dpe}). 
%We call a diffeomorphism $f$ partially hyperbolic if there exists an invariant continuous splitting $TM=E^u\oplus E^c\oplus E^s$ such that $Df|_{E^u}$ is a uniform expansion, $Df|_{E^s}$ is a uniform contraction, and $Df|_{E^c}$ lies in between them: 
%$$
%\frac{\|Df(x)(v^s)\|}{\|Df(x)(v^c)\|}\le \frac{1}{2} \quad \textrm{and} \quad \frac{\|Df(x)(v^c)\|}{\|Df(x)(v^u)\|}\le \frac{1}{2}
%$$
%for any unit vectors $v^s\in E^s(x)$, $v^c\in E^c(x)$, $v^u\in E^u(x)$ and $x\in M$.
For a $C^2$ partially hyperbolic diffeomorphism exhibiting strong unstable direction $E^u$, any physical measure has the property that its conditional measures along strong unstable leaves are absolutely continuous with respect to the corresponding Lebesuge measures. We call measures with this property as \emph{Gibbs $u$-states}, a notion borrowed from Pesin, Sinai \cite{ps82}.
Observe that Gibbs $u$-states include any existing physical measures, while not all Gibbs $u$-states are physical measures.
%It is reasonable to ask that when the Gibbs $u$-state is physical.

There is a sequence of works that focus on finding physical measures under the hyperbolicity of Gibbs $u$-states.
For a partially hyperbolic diffeomorphism with splitting $E^u\oplus E^c \oplus E^s$, we say $E^c$ is mostly contracting (resp. mostly expanding) if all the central Lyapunov exponents of every Gibbs $u$-state are negative (resp. positive).
When $E^c$ is mostly expanding, Bonatti-Viana \cite[Theorem A]{bv} proved the existence of finitely many physical measures, such that Lebesgue almost every point is contained in the basins of these physical measures, which we call basin covering property. In a similar spirit, 
Andersson-Vasqueze \cite[Theorem C]{AV15} showed the same result when $E^c$ is mostly contracting. Mi-Cao-Yang \cite[Theorem A]{MCY17} considered the case when $E^c$ admits the sub-splitting $E^{cu}\oplus_{\succ} E^{cs}$ such that $E^{cu}$ is mostly expanding and $E^{cs}$ is mostly contracting, they obtained the same result on physical measures.

All of the above results presuppose the existence of hyperbolic Gibbs $u$-state with definite sign on the center direction. In a recent work, Hua-Yang-Yang \cite[Corollary E]{HYY19} studied the partially hyperbolic with one-dimensional center, they obtained the finitely many physical measures with basin covering property when all the Gibbs $u$-states are hyperbolic, i.e., every Gibbs $u$-state admits nonzero central Lyapunov exponents. What is case for partially hyperbolic diffeomorphism with higher dimension center?

%There are many progress in this direction, we list some of them below:
%
%\begin{itemize}
%\item Bonatti, Viana \cite[Theorem A]{bv} obtain the existence and finiteness of physical measures whose basins cover a full Lebesgue measure subset by assuming all the Gibbs $u$-states admits only negative central Lyapunov exponents.
%\item Under the hyperbolicity of Gibbs $u$-state, for the case that ${\rm dim}E^c=1$, it is shown in \cite{HYY19} that there exist finitely many physical measures whose basins cover a full Lebesgue measure subset. 
%\item In \cite{AV15}, the authors proved the existence and finiteness of physical measures for partially hyperbolic diffeomorphisms when the central Lyapunov exponents of Gibbs $u$-states are or all positive.
%\item Burns-Dolgopyat-Pesin-Pollicott \cite{bdpp} got the uniqueness of physical measure with basin covering property. under the condition that there is a Gibbs $u$-state, which admits only negative central Lyapunov exponents on a set of positive measure, together with the denseness of unstable manifolds. 
%\end{itemize}

%We begin from the following simple example: 

Consider $g$ as a linear transtive Anosov diffeomorphism on tours $\mathbb{T}^2$ having eigenvalues $0<\lambda_1<1<\lambda_2$. Let $h$ be a Morse-Smale diffeomorphism on sphere $\mathbb{S}^2$ with a sink $p$ of eigenvalue $0<\lambda_1<\mu_1$ and a single source $q$ of eigenvalue $1<\mu_2<\lambda_2$. Now consider $f: \mathbb{T}^2\times \mathbb{S}^2\to \mathbb{T}^2\times \mathbb{S}^2$ defined by $f=g\times h$.
Then $f$ is a partially hyperbolic diffeomorphism with $E^c$ tangent to the fibers $\{x\}\times \mathbb{S}^2$. Let $\mu$ be the unique physical measure of $g$ on $\mathbb{T}^2$. Thus, $\mu\times \delta_p$ and $\mu\times \delta_q$ are the two Gibbs $u$-states of $f$. Moreover, one can check that $\mu\times \delta_p$ is a physical measure of $f$, whose basin covers a full Lebesgue measure of $M$.

The main feature of the above example is the coexistence of Gibbs $u$-states with different signs on central Lyapunov exponents.
Inspired by the above observation, it makes sense to introduce the following notion. 

\begin{definition}
Assume that $f$ is a $C^2$ diffeomorphism with a partially hyperbolic splitting 
$
TM=E^u\oplus_{\succ} E^c\oplus_{\succ} E^s.
$
We say $E^c$ is \emph{u-definite} if for every ergodic Gibbs $u$-state $\mu$,
\begin{itemize}
\item either, all the Lyapunov exponents of $\mu$ along $E^c$ are positive,
\item or, all the Lyapunov exponents of $\mu$ along $E^c$ are negative.
\end{itemize}
\end{definition}

The main result of this paper is the following: 

\begin{theoremalph}\label{TheoA}
Let $f$ be a $C^2$ diffeomorphism with a partially hyperbolic splitting 
$
TM=E^u\oplus_{\succ} E^c\oplus_{\succ} E^s.
$
If $E^c$ is u-definite, then there are finitely many physical measures, whose union of basins cover a full Lebesgue measure subset of $M$.
\end{theoremalph}

Different to the previous settings in \cite{bv,AV15}, where all the central Lyapunov exponents of Gibbs $u$-states have the same sign, the $u$-definite property on $E^c$ dose allow Gibbs $u$-states to present the mixed behavior, i.e., different Gibbs $u$-states can possess different signs on central Lyapunov exponents. Thus, Theorem \ref{TheoA} can be seen as an extension of \cite{bv, AV15}. 

We would like to point out that when ${\rm dim}E^c=1$ as in \cite{HYY19}, the hyperbolicity of Gibbs $u$-states is equivalent to the $u$-definite property on $E^c$. When ${\rm dim}E^c\ge 2$, the following counter-example
indicates that just the hyperbolicity of Gibbs $u$-states can not deduce the conclusion of physical measures as in Theorem \ref{TheoA}.

Let $g$ be the time one map of the flow exhibiting the attractor of ``Bowen's eye", which consisting of a double saddle connection between two saddles $A$ and $B$ (see Figure \ref{k}). Now take $f$ as the product of $g$ with a transitive Anosov diffeomorphism $h$. By taking eigenvalues of $A, B$ suitably, $f$ will be a partially hyperbolic diffeomorphism with following properties (see \cite[Example 1.6]{v1}):
\begin{itemize}
\item $f$ has two ergodic hyperbolic Gibbs $u$-states, each of them admits central Lyapunov exponents with opposite signs.
\smallskip
\item $f$ has no physical measures.
\end{itemize}

\begin{figure}[h]\label{k}
	\centering
	\includegraphics[width=0.60\textwidth]{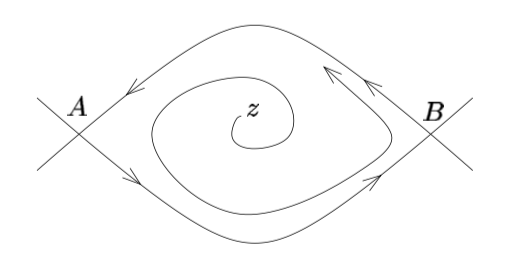}
	\caption{``Bowen's eye".}
\end{figure}

%Indeed, we may give a counterexample of higher dimensional central by considering the product of Anosov diffeomorphism with the diffeomorphism exhibiting the attractor of Bowen's eye without physical measures, where there exists some ergodic Gibbs $u$-state admitting central Lyapunov exponents with different signs. 

By applying the absolute continuity of Pesin stable lamination (e.g. see \cite[$\S$ 8.6]{bp07}), one knows that each ergodic Gibbs $u$-state with negative central Lyapunov exponents must be a physical measure, which is just the setting in \cite{bv}. The main difficulty in the proof of Theorem \ref{TheoA} is to select the remaining physical measures from Gibbs $u$-states having positive central Lyapunov exponents. We solve this issue by constructing \emph{Gibbs $cu$-states}--the invariant measures whose disintegration on Pesin unstable manifolds tangent to center-unstable direction $E^u\oplus E^c$ are absolutely continuous w.r.t. Lebesgue measures on these manifolds. This has been investigated in the previous works \cite{AV15, MCY17}, where the Gibbs $cu$-states were built in a geometrical way. Both in \cite{AV15, MCY17}, the main technical step is to get non-uniform expansion on center from the compactness of the set of Gibbs $u$-states. This ensures that they can apply the techniques from \cite{ABV00, AD} to construct Gibbs $cu$-states. We emphasize that due to the absence of compactness of the set of Gibbs $u$-states having positive central Lyapunov exponents, the method from \cite{AV15, MCY17} cannot be adapted to our setting. We take a different strategy, we follow a variational approach by utilizing the entropy formulas to build Pesin entropy formula, similar to \cite{HYY19}. Additionally, some ideas from the theory of pseudo-physical measures will be involved as well.

\section{Preliminary}

\subsection{Basic notations}\label{e}
Throughout, we will use $M$ to denote the compact Riemannian manifold and ${\rm Leb}$ to denote the Lebesgue measure of $M$ given by the Riemmannian metric. Given a smooth sub-manifold $D$, denote by ${\rm Leb}_D$ the induced Lebesgue measure in $D$.

Let $\mathscr{M}$ be the space of Borel probability measures on $M$ endowed with the weak$^*$-topology. Since $M$ is a compact metric space, this topology is metrizable. We denote by $dist$ a metric on $\mathscr{M}$.
%Let us define a distance ${\rm dist}$ on $\mathscr{M}$ as follows: consider $\{\varphi_n\}_{n\in \NN}$ as a countable dense subset of the space of continuous functions on $M$, 
%$$
%{\rm dist}(\mu,\nu)=\sum_{n\in \NN}\frac{|\int \varphi_n d\mu- \int \varphi_n d\nu|}{2^n\sup_{x\in M}|\varphi_n(x)|}, \quad \forall~ \mu, \nu \in \mathscr{M}.
%$$
%The $``{\rm dist}"$ gives the weak$^*$-topology on $\mathscr{M}$.

Let $f$ be a diffeomorphism on $M$.
For $x\in M$, denote by $\mathscr{M}(x)$ the set of limit measures of 
$
\{\frac{1}{n}\sum_{i=0}^{n-1}\delta_{f^i(x)}\}_{n\in \NN}.
$
In other words, we have $\mu\in \mathscr{M}(x)$ if there exists a subsequence $\{n_i\}$ tends to infinity such that 
$$
\frac{1}{n_i}\sum_{j=0}^{n_i-1}\delta_{f^j(x)}\xrightarrow[i\to +\infty]{weak^*} \mu.
$$

\subsection{Dominated splitting and partially hyperbolicity}\label{dpe}
Let $f$ be a $C^1$ diffeomorphism on $M$. 
A $Df$-invariant splitting $TM=E\oplus_{\succ} F$ of the tangent bundle is called a \emph{dominated splitting} if $E$ dominates $F$, i.e., there exist $C>0$ and $0<\lambda<1$ such that for every $x\in M$,
$$
\|Df^n|_{F(x)}\|\cdot \|Df^{-n}|_{E(f^n(x))}\|<C\lambda^n, \quad \forall n\in \NN.
$$

%Given $a > 0$, define the cone field $\mathcal{C}_a^{E}=\{\mathcal{C}_a^{E}(x)\}_{x\in M}$ associated to $E$ of width $a$ by
%$$
%\mathcal{C}_a^E(x)=\{v=v^E+v^F\in E(x)\oplus F(x): \|v^F\|\le a \|v^E\|\}
%$$
%for every $x\in M$. By symmetry, we can define the cone field $\mathcal{C}_a^F$ associated to $F$.
%We say a smooth embedded sub-manifold $D$ is tangent to $E$ if $T_xD=E(x)$ for every $x\in D$.

One says that $f$ is a \emph{partially hyperbolic} diffeomorphism if there exists a dominated splitting 
$TM=E^u\oplus_{\succ} E_1\oplus_{\succ}\cdots \oplus_{\succ} E_k\oplus_{\succ} E^s$ 
on the tangent bundle such that
\begin{itemize}
\item $E^u$ is uniformly expanding and $E^s$ is uniformly contracting.
\item at least one of $E^u$ and $E^s$ is nontrivial.
\end{itemize}
 
For any fixed small $\delta>0$, denote by $W^u_{\delta}(x)$ the local unstable manifold (tangent to $E^u$) of $x\in M$ with size $\delta$. For each $x\in M$, define
$$
D_{n}^u(x):=f^{-n}(W^u_{\delta}(f^n(x))), \quad n\in \NN.
$$
We will call $\{D_{n}^u(x)\}_{n\in \NN}$ the \emph{unstable density basis} of $x$, mainly due to the following result(see \cite[Proposition 3.1]{Xia} or \cite[Theorem 2.18]{CYZ18}).

\begin{lemma}\label{ubd}
Let $f$ be a diffeomorphism with partially hyperbolic splitting $TM=E^u\oplus_{\succ} E^{cs}$. If $D$ is an unstable disk satisfying ${\rm Leb}_D(A)>0$ for some measurable set $A$, then for ${\rm Leb}_D$-a.e. $x\in A\cap D$ we have
\begin{equation}\label{und}
\lim_{n\to +\infty}\frac{{\rm Leb}_D(D_{n}^u(x)\cap A)}{{\rm Leb}_D(D_{n}^u(x))}=1.
\end{equation}
\end{lemma}
We say that $x$ is an \emph{unstable Lebesgue density} point if it satisfies (\ref{und}).

\subsection{Gibbs $u$-state and Gibbs $cu$-state}
%Gibbs $u$-states are important candidates for SRB measures and physical measures. 

%Let $\mathscr{M}$ be the space of Borel probability measures on $M$.
%Let us define a distance ${\rm dist}$ on $\mathscr{M}$ as follows: consider $\{\varphi_n\}_{n\in \NN}$ as a countable dense subset of the space of continuous functions on $M$, 
%$$
%{\rm dist}(\mu,\nu)=\sum_{n\in \NN}\frac{|\int \varphi_n d\mu- \int \varphi_n d\nu|}{2^n\sup_{x\in M}|\varphi_n(x)|}, \quad \forall~ \mu, \nu \in \mathscr{M}.
%$$
%The $``{\rm dist}"$ gives the weak$^*$-topology on $\mathscr{M}$.
%where $\{\psi\}_{n\in \NN}$ is a countable dense subset of the space of continuous functions on $M$. Note that all the convergences of measures in this paper is in weak*-topology.

%Let $f$ be a diffeomorphism on $M$.
%For $x\in M$, denote by $\mathscr{M}(x)$ the set of limit measures of 
%$
%\{\frac{1}{n}\sum_{i=0}^{n-1}\delta_{f^i(x)}\}_{n\in \NN}.
%$
%In other words, we have $\mu\in \mathscr{M}(x)$ if there exists a subsequence $\{n_i\}$ tends to infinity such that 
%$$
%\frac{1}{n_i}\sum_{j=0}^{n_i-1}\delta_{f^j(x)}\xrightarrow[i\to +\infty]{weak^*} \mu.
%$$

Let $f$ be a $C^2$ diffeomorphism with partially hyperbolic splitting $TM=E^u\oplus_{\succ} E^{cs}$. 
Recall that an $f$-invariant measure $\mu$ is a Gibbs $u$-state if its conditional measures along strong unstable manifolds are absolutely continuous w.r.t. Lebesgue measures on these manifolds. Let us collect several properties of Gibbs $u$-states, see \cite[Subsection 11.2]{BDV05} for details.

\begin{proposition}\label{uc}
Let $f$ be a diffeomorphism with partially hyperbolic splitting $TM=E^u\oplus_{\succ} E^{cs}$, one has the following properties:
\begin{enumerate}
\item\label{11} The ergodic components of any Gibbs $u$-state are Gibbs $u$-states.
\item\label{22} The set of Gibbs $u$-states is compact in the weak$^*$-topology.
\item\label{33} For Lebesgue almost every $x\in M$, any measure in $\mathscr{M}(x)$ is a Gibbs $u$-state. 
\end{enumerate}
\end{proposition}

%Assume that $f$ is a partially hyperbolic diffeomorphism with dominated splitting $TM=E^u\oplus_{\succ} E^{cs}$. 
%For any fixed small $\delta>0$, denote by $W^u_{\delta}(x)$ the local unstable manifold of $x\in M$ with size $\delta$. For any $x\in M$ define  
%$$
%D_{n}^u(x):=f^{-n}(W^u_{\delta}(f^n(x))), \quad n\in \NN.
%$$
%We will call $\{D_{n}^u(x)\}_{n\in \NN}$ the unstable density basis of $x$, mainly due to the following result(see \cite[Proposition 3.1]{Xia} or \cite[Theorem 2.18]{CYZ18}). 
%
%
%\begin{lemma}\label{ubd}
%Let $f$ be a diffeomorphism with partially hyperbolic splitting $TM=E^u\oplus_{\succ} E^{cs}$. If $D$ is an unstable disk satisfying ${\rm Leb}_D(A)>0$ for some measurable set $A$, then for ${\rm Leb}_D$-a.e. $x\in A\cap D$ we have
%\begin{equation}\label{und}
%\lim_{n\to +\infty}\frac{{\rm Leb}_D(D_{n}^u(x)\cap A)}{{\rm Leb}_D(D_{n}^u(x))}=1.
%\end{equation}
%\end{lemma}
%We say $x$ satisfying (\ref{und}) an \emph{unstable Lebesgue density} point.

%Let us recall the following entropy formula for diffeomorphisms with dominated splittings, see \cite[Theorem F]{CYZ18}, \cite[Theorem A]{cce15} for the proof.
%
%\begin{proposition}\label{cug}
%Let $f$ be a diffeomorphism with dominated splitting $TM=E\oplus_{\succ} F$. Then for Lebesgue almost every $x\in M$, any limit measure $\mu\in\mathscr{M}(x)$ satisfies 
%$$
%h_{\mu}(f)\ge \int \log |{\rm det}Df|_{E}| d \mu.
%$$
%\end{proposition}

Recall another notion called \emph{Gibbs $cu$-states}.
Let $f$ be a diffeomorphism with partially hyperbolic splitting $TM=E\oplus_{\succ} F$. An invariant measure $\mu$ called a Gibbs $cu$-state (of $E$) if
\begin{itemize}
\item all the Lyapunov exponents of $\mu$ along $E$ are positive, 
\item the conditional measure of $\mu$ along Pesin unstable manifolds tangent to $E$ are absolutely continuous w.r.t. Lebesgue measure on these manifolds.
\end{itemize}

We have the following result. See a proof in \cite[Lemma 2.4]{cv} or \cite[Proposition 4.7]{CMY22}.
\begin{proposition}\label{cue}
Assume that $f$ is a $C^2$ diffeomorphism with a dominated splitting $TM = E\oplus_{\succ} F$. If $\mu$ is a Gibbs $cu$-state of $E$, then almost every ergodic component of $\mu$ is also a Gibbs $cu$-state of $E$.
\end{proposition}

%\begin{proposition}\label{up1}
%If $\mu$ is an ergodic Gibbs $u$-state, then there exists an unstable disk $\cD$ with any small size such that 
%\begin{itemize}
%\item
%\item
%\end{itemize}
%\end{proposition}

\subsection{Pseudo-physical measures}
Given a homeomorphism $f$ on $M$.
An $f$-invariant measure $\mu$ called a \emph{pseudo-physical measure} if 
$$
{\rm Basin}_{\varepsilon}(\mu)=\{x\in M: {\rm dist}(\mathscr{M}(x),\mu)<\varepsilon\}
$$
has positive Lebesgue measure for any $\varepsilon>0$. 

The notion of pseudo-physical measure was introduced by Catsigeras and Enrich \cite{ce11} to generalize physical measures. By definition we know that every physical measure is also a pseudo-physical measure.

\begin{lemma}\label{ps1}\cite[Theorem 1.5]{ce11}
For every homeomorphism $f$ on $M$. The set of pseudo-physical measures of $f$ is nonempty, which is a minimal compact invariant subset containing $\mathscr{M}(x)$ for Lebesuge almost every point $x\in M$.
\end{lemma}

\begin{remark}\label{hh}
From Lemma \ref{ps1} and property (\ref{33}) of Proposition \ref{uc}, we know that every pseudo-physical measure is a Gibbs $u$-state for any partially hyperbolic diffeomorphism exhibiting $E^u$.
\end{remark}

\begin{lemma}\label{ps2}\cite[Theorem 1]{cce15}
Assume that $f$ is a diffeomorphism with dominated splitting $TM=E\oplus_{\succ} F$. Then any pseudo-physical measure $\mu$ satisfies 
$$
h_{\mu}(f)\ge \int \log |{\rm det}Df|_{E}|d\mu.
$$
\end{lemma}

As an immediate consequence of Lemma \ref{ps2} and Lemma \ref{ps1}, one also has the next result:

\begin{corollary}\label{cug}
Let $f$ be a diffeomorphism with dominated splitting $TM=E\oplus_{\succ} F$. Then for Lebesgue almost every $x\in M$, any limit measure $\mu\in\mathscr{M}(x)$ satisfies 
$$
h_{\mu}(f)\ge \int \log |{\rm det}Df|_{E}| d \mu.
$$
\end{corollary}

\section{Proof of Theorem \ref{TheoA}}
Now we assume that $f$ is a $C^{2}$ diffeomorphism with partially hyperbolic splitting $TM=E^u\oplus_{\succ} E^c\oplus_{\succ} E^s$. 
Given an ergodic measure $\mu$, denote by $\lambda^{c}_{\max}(\mu)$ the maximal Lyapunov exponent of $\mu$ along $E^c$, that is
$$
\lambda^{c}_{\max}(\mu)=\lim_{n\to +\infty}\frac{1}{n}\log \|Df^n|_{E^c(x)}\|, \quad \mu\textrm{-a.e.}~x\in M.
$$

In what follows, we will find finitely physical measures with basin covering property from the set of ergodic Gibbs $u(cu)$-states. 
Denote by $G^u_{erg}(f)$ the set of ergodic Gibbs $u$-states of $f$.
Let $\mathcal{A}$ be the set of ergodic Gibbs $u$-states with negative central Lyapunov exponents, i.e.,
$$
\mathcal{A}=\{\mu \in G_{erg}^u(f): \lambda^{c}_{\max}(\mu)<0\}.
$$
Let $\mathcal{B}$ be the set of ergodic Gibbs $cu$-states. 

By applying the absolute continuity of (Pesin) stable lamination, we conclude that every measure of $\mathcal{A}\cup \mathcal{B}$ is a physical measure. Moreover, we have the next observation.
\begin{theorem}\label{pre}
Let $f$ be a $C^2$ diffeomorphism with a partially hyperbolic splitting 
$
TM=E^u\oplus_{\succ} E^c\oplus_{\succ} E^s.
$
If $E^c$ is $u$-definite, then there exist physical measures contained in $\mathcal{A}\cup \mathcal{B}$.
\end{theorem}

\begin{proof}
From properties (\ref{11}) and (\ref{33}) of Proposition \ref{uc}, we have the existence of ergodic Gibbs $u$-states. If there exists some ergodic Gibbs $u$-state $\mu$ with negative central Lyapunov exponents, that is, $\mu\in \mathcal{A}$, then it is a physical measure. Otherwise, since $E^c$ is $u$-definite, all the ergodic Gibbs $u$-states must admit only positive Lyapunov exponents along $E^c$. In this case, $E^c$ is mostly expanding for $f$. According to \cite[Theorem A]{AV15}, it follows that there are finitely many physical measures, which are ergodic Gibbs $cu$-states, so are in $\mathcal{B}$. Up to now, we have shown the desired result.
\end{proof}

In view of Theorem \ref{pre}, to complete the proof of Theorem \ref{TheoA}, it suffices to prove the following two theorems.
\begin{theorem}\label{covering}
Let $f$ be a $C^2$ diffeomorphism with a partially hyperbolic splitting 
$
TM=E^u\oplus_{\succ} E^c\oplus_{\succ} E^s.
$
If $E^c$ is $u$-definite, then
$$
{\rm Leb}\left(M\setminus \bigcup_{\mu\in \mathcal{A}\cup \mathcal{B}}{\rm Basin}(\mu)\right)=0.
$$
\end{theorem}

\begin{theorem}\label{fin}
Let $f$ be a $C^2$ diffeomorphism with a partially hyperbolic splitting 
$
TM=E^u\oplus_{\succ} E^c\oplus_{\succ} E^s.
$
If $E^c$ is $u$-definite, then $\# (\mathcal{A}\cup \mathcal{B})<+\infty$.
\end{theorem}

We will give the proofs of Theorem \ref{covering} and Theorem \ref{fin} in next two subsections.

\subsection{Proof of Theorem \ref{covering}}

\begin{lemma}\label{ccu}
Let $f$ be a $C^2$ diffeomorphism with a partially hyperbolic splitting 
$
TM=E^u\oplus_{\succ} E^c\oplus_{\succ} E^s.
$
If $E^c$ is $u$-definite and ${\rm Leb}(M\setminus \cup_{\nu\in \mathcal{A}}{\rm Basin}(\nu))>0$, then for Lebesgue almost every $x\in M\setminus \cup_{\nu\in \mathcal{A}}{\rm Basin}(\nu)$, any $\mu\in \mathscr{M}(x)$ admits only positive Lyapunov exponents along $E^c$. 
\end{lemma}

\begin{proof}
By property (\ref{22}) of Proposition \ref{uc}, one can take a subset $\mathcal{R}^u$ with full Lebesgue measure, such that for any $x\in \mathcal{R}^u$, any measure of $\mathscr{M}(x)$ is a Gibbs $u$-state. 
%By property (\ref{22}) of Proposition \ref{uc}, $\mathcal{R}^u$ exhibits full Lebesgue measure. 
Arguing by contradiction, recalling that $E^c$ is $u$-definite,
there exists a subset $\mathscr{A}\subset \mathcal{R}^u\setminus \cup_{\nu\in \mathcal{A}}{\rm Basin}(\nu)$ with positive Lebesgue measure such that for every $x\in \mathscr{A}$, there exists $\mu\in \mathscr{M}(x)$ having some negative Lyapunov exponents along $E^c$. 

%Consequently, there is some ergodic component $\nu$ of $\mu$ with following properties:
%\begin{enumerate}[(I)]
%\item\label{cp1} $\nu$ admits only negative Lyapunov exponents along $E^c$, i.e., 
%$$
%\lambda_{\max}^c(\nu)<0;
%$$
%\item\label{cp2} ${\rm supp}(\nu)\subset {\rm supp}(\mu).$
%\end{enumerate}

\smallskip
Take an open subset $U$ with ${\rm Leb}(U\cap \mathscr{A})>0$. By foliating $U$ with unstable disks and applying the absolute continuity of unstable foliation we can choose an unstable Lebesgue density point $x$ of $\mathscr{A}$ in some unstable disk $D$. Fix $\delta<{\rm diam}D$. By the choice of $\mathscr{A}$, we can fix $\mu\in \mathscr{A}$ for which there are negative Lyapunov exponents along $E^c$. 
%$$
%\frac{{\rm Leb}(B_D(x_0,r)\cap A)}{{\rm Leb}(B_D(x_0,r))}\xrightarrow[r\to 0] ~1,
%$$
%where $B_D(x_0,r)=\{y\in D: d_D(x_0,y)<r\}$. 
By ergodic decomposition theorem \cite[Theorem 6.4]{Man87}, we have 
$$
\mu({\rm supp}(\mu))=\int \mu_z({\rm supp}(\mu))d\mu(z),
$$ 
which implies that ${\rm supp}(\mu_z)\subset{\rm supp}(\mu)$ for $\mu$-a.e.~$z$. Then, in view of the central Lyapunov exponents of $\mu$, one can find an ergodic component $\nu$ of $\mu$ satisfying
\begin{enumerate}[(i)]
\item\label{cp1} $\lambda_{\max}^c(\nu)<0,$
\smallskip
\item\label{cp2} ${\rm supp}(\nu)\subset {\rm supp}(\mu).$
\end{enumerate} 
As $\mu$ is a Gibbs $u$-state, $\nu$ is in $\mathcal{A}$ by Property (\ref{11}) of Proposition \ref{uc}.
Observe that the conditional measures of $\mu$ along unstable manifolds are absolutely continuous w.r.t. the Lebesgue measures. This together with property (\ref{cp1}) and the ergodicity of $\nu$ imply that there exists an unstable disk $\mathcal{D}^u$ with following properties:
\begin{itemize}
\item ${\rm diam}(\mathcal{D}^u)<\delta/2$,
\smallskip
\item $\mathcal{D}^u\subset {\rm supp}(\nu)$,
\smallskip
\item ${\rm Leb}_{\mathcal{D}^u}$-almost every point of $\mathcal{D}^u$ is contained in ${\rm Basin}(\nu)$ and admits the local Pesin unstable manifold.
\end{itemize}
The last property implies that there exists a subset $\D\subset \mathcal{D}^u\cap {\rm Basin}(\nu)$ satisfying ${\rm Leb}_{\cD^u}(\D)>0$ and every point of $\D$ admits local Pesin stable manifold with uniform size.

%\begin{itemize}
%\item there exists a subset $\D\subset \mathcal{D}^u\cap B(\nu)$ with ${\rm Leb}_{\cD^u}(\D)>0$ so that every point of $\D$ admits local Pesin stable manifold with uniform size.
%\end{itemize}

According to the choice of $\mathcal{D}^u$ above, by considering open tubular neighborhood of $\D^u$, one can construct an open neighborhood $B$ containing $\mathcal{D}^u$ such that 
\begin{itemize}
\item for every point $y\in B$, the unstable disk of $y$ with radius $\delta$ intersects every local Pesin unstable manifold from $\D$ transversely. 
\end{itemize}
By property (\ref{cp2}), we obtain 
$\mu(B)>0$. Due to $\mu\in \mathscr{M}(x)$, there is a subsequence $\{n_k\}\subset \NN$ such that $\frac{1}{n_k}\sum_{i=0}^{n_k-1}\delta_{f^i(x)}$ converges in the weak$^*$-topology to some $\mu$, as $k\to +\infty$.
%\begin{equation}\label{con1}
%\frac{1}{n_k}\sum_{i=0}^{n_k-1}\delta_{f^i(x)}\xrightarrow[k\to +\infty]{weak*} \mu,
%\end{equation}
This gives that 
$$
\liminf_{k\to +\infty}\frac{1}{n_k}\sum_{i=0}^{n_k-1}\chi_{B}(f^i(x))\ge \mu(B)>0.
$$
As a consequence, there exists an increasing sequence $\{\ell_k\}_{k\in \NN}$ of integers tending to infinite such that 
$$
f^{\ell_k}(x)\in B, \quad \forall k\in \NN.
$$ 

For each $k\in \NN$,  denote by $D^u(f^{\ell_k}(x),\delta)$ the unstable disk of $f^{\ell_k}(x)$ with radius $\delta$. By the construction of $B$, it follow that $D^u(f^{\ell_k}(x),\delta)$ intersects every local Pesin stable manifold from $\D$ transversely. By applying the absolute continuity of Pesin stable lamination, one obtains a constant $\alpha>0$ satisfying
\begin{equation}\label{fp}
\frac{{\rm Leb}_{D^u(f^{\ell_k}(x),\delta)}(D^u(f^{\ell_k}(x),\delta)\cap {\rm Basin}(\nu))}{{\rm Leb}_{D^u(f^{\ell_k}(x),\delta)}(D^u(f^{\ell_k}(x),\delta))}>\alpha,\quad \forall k\in \NN.
\end{equation}
Put 
$$
D_k(x):=f^{-\ell_k}(D^u(f^{\ell_k}(x),\delta)) \quad \textrm{for every}~ k\in \NN.
$$
Since the unstable disks are uniformly contracted, by the standard bounded distortion argument, (\ref{fp}) yields a constant $\beta>0$ such that 
\begin{equation}\label{bd2}
\frac{{\rm Leb}_{D_k(x)}(D_k(x)\cap {\rm Basin}(\nu))}{{\rm Leb}_{D_k(x)}(D_k(x))}>\beta,\quad \forall k\in \NN.
\end{equation}
On the other hand, note that $\{D_k(x)\}_{k\in \NN}$  is the unstable density basis of $x$, from the choice of $x$ and Lemma \ref{ubd}, it follows that
\begin{equation}\label{bd3}
\frac{{\rm Leb}_{D_k(x)}(D_k(x)\cap \mathscr{A})}{{\rm Leb}_{D_k(x)}(D_k(x))}\xrightarrow [k\to +\infty]~1.
\end{equation}
Combining (\ref{bd2}) and (\ref{bd3}), we have for sufficiently large $k$ that
$$
{\rm Leb}_{D_k(x)}({\rm Basin}(\nu)\cap \mathscr{A})>0,
$$
which contradicts with the definition of $\mathscr{A}$, recalling that $\nu\in \mathcal{A}$
\end{proof} 

As a consequence of Lemma \ref{ccu}, together with the argument on entropy formulas we have

\begin{proposition}\label{cup}
Under the assumption of Lemma \ref{ccu}, we have that for Lebesgue almost every $x\in M\setminus \cup_{\nu\in \mathcal{A}}{\rm Basin}(\nu)$, any $\mu\in \mathscr{M}(x)$ is a Gibbs $cu$-state.
%$$
%\lim_{n\to +\infty}\frac{1}{n}\log {\rm m}\left(Df^n|_{E^c(x)}\right)>0,\quad \mu-\textrm{a.e.}~x\in M.
%$$

\end{proposition}
\begin{proof}
By Lemma \ref{ccu}, for Lebesgue almost every $x\in M\setminus \cup_{\nu\in \mathcal{A}}{\rm Basin}(\nu)$, any $\mu\in \mathscr{M}(x)$ admits only positive Lyapunov exponents along $E^c$. By Ruelle's inequality, one has
$$
h_{\mu}(f)\le \int \log |{\rm det} Df|_{E^u\oplus E^c}| d\mu.
$$
On the other hand, Corollary \ref{cug} implies that by reducing to a smaller full Lebesgue measure subset, we have also that
$$
h_{\mu}(f)\ge \int \log |{\rm det} Df|_{E^u\oplus E^c}| d\mu.
$$
Therefore,
$$
h_{\mu}(f)= \int \log |{\rm det} Df|_{E^u\oplus E^c}| d\mu.
$$
The classical result \cite[Theorem A]{ly} tells us that $\mu$ is a Gibbs $cu$-state of $E^u\oplus E^c$.
\end{proof}

Now we can give the proof of Theorem \ref{covering}.

\begin{proof}[Proof of Theorem \ref{covering}] 
%If $\mathcal{A}=\emptyset$, then all the Gibbs $u$-stats admits only positive Lyapunov exponents along $E^c$, thus $E^c$ is mostly expanding by $f$. By \cite[Theorem A]{AV15}, $f$ admits finitely many ergodic Gibbs $cu$-states, they are physical measures whose basins cover a full Lebesgue measure subset of $M$. That is
%$$
%{\rm Leb}\left(M\setminus \bigcup_{\mu\in \mathcal{B}}\mathcal{B}(\mu,f)\right)=0.
%$$
%
Assume by contradiction that
$$
{\rm Leb}\left(M\setminus \bigcup_{\nu\in \mathcal{A}\cup \mathcal{B}}{\rm Basin}(\nu)\right)>0.
$$
In particular, we have 
$$
{\rm Leb}\left(M\setminus \bigcup_{\nu\in \mathcal{A}}{\rm Basin}(\nu)\right)>0.
$$ 
%Hence, we have
%$$
%{\rm Leb}\left(M\setminus \bigcup_{\mu\in \mathcal{A}}{\rm Basin}(\mu)\right)\ge {\rm Leb}\left(M\setminus \bigcup_{\mu\in \mathcal{A}\cup \mathcal{B}}{\rm Basin}(\mu)\right)>0.
%$$
By Proposition \ref{cup}, there exists a subset $\mathscr{N}\subset M\setminus \bigcup_{\nu\in \mathcal{A}\cup \mathcal{B}}{\rm Basin}(\nu)$ with positive Lebesgue measure such that for any $x\in \mathscr{N}$, any $\mu\in \mathscr{M}(x)$ is a Gibbs $cu$-state.
Now we take $x$ as a Lebesgue density point of $\mathscr{N}$. Fix any $\mu\in \mathscr{M}(x)$, thus there exists a subsequence $\{n_k\}$ of integers such that 
\begin{equation}\label{con}
\frac{1}{n_k}\sum_{i=0}^{n_k-1}\delta_{f^i(x)}\xrightarrow[k\to +\infty]{weak^*} \mu.
\end{equation}

As $\mu$ is a Gibbs $cu$-state, we have by Proposition \ref{cue} that almost every ergodic component of $\mu$ is a Gibbs $cu$-state, which is also physical.
Observe also that the support of almost every ergodic component of $\mu$ is contained in the support of $\mu$, as presented in the proof of Lemma \ref{ccu}.
Consequently, we can choose an ergodic component $\nu$ of $\mu$ such that 
\begin{itemize}
\item $\nu$ is an ergodic Gibbs $cu$-state, so is in $\mathcal{B}$,
\smallskip
\item ${\rm supp}(\nu)\subset {\rm supp}(\mu)$.
\end{itemize}

Let us take a local Pesin unstable disk $\gamma\subset {\rm supp}(\nu)$ tangent to $E^u\oplus E^c$ such that ${\rm Leb}_{\gamma}$-almost every point of $\gamma$ is contained in ${\rm Basin}(\nu)$. Consider the open subset $\mathcal{U}$ formed by local strong stable manifolds from $\gamma$ as follows
$$
\mathcal{U}=\bigcup_{y\in \gamma}W^s_{loc}(y),
$$
where $W^s_{loc}(y)$ denotes the local strong stable manifold of $y$.
Note that local strong stable manifold can only lie in the same basin. By the absolute continuity of stable lamination, one gets that Lebesgue almost every point of $\mathcal{U}$ is contained in the basin of $\nu$. 

Since ${\rm supp}(\nu)\subset {\rm supp}(\mu)$, we obtain $\nu(\mathcal{U})>0$ by construction. The convergence (\ref{con}) implies that there exists $m\in \NN$ such that 
$f^{m}(x)\in \mathcal{U}$. From the continuity of $f$, we know that there exists an open ball $B_x$ centered at $x$ such that $f^{m}(B_x)\subset \mathcal{U}$. By invariance of the basin ${\rm Basin}(\nu)$ we then have 
$$
{\rm Leb}(B_x\setminus {\rm Basin}(\nu))=0.
$$ 
However, since $x$ is a Lebesgue density of $\mathscr{N}$, one has that 
${\rm Leb}(B_x\cap \mathscr{N})>0$. By construction of $\mathscr{N}$, we then have 
$$
{\rm Leb}(B_x\setminus {\rm Basin}(\nu))>0,
$$
which is a contradiction and the proof is complete.
\end{proof}

\subsection{Proof of Theorem \ref{fin}}

%We divide Theorem \ref{fin} by the following two propositions.
%
%\begin{proposition}\label{e1}
%$\#\mathcal{A}<+\infty$.
%\end{proposition}
%
%\begin{proposition}\label{e2}
%$\#\mathcal{B}<+\infty$.
%\end{proposition}

We recall the following observation on Gibbs $cu$-states, see a proof in \cite[Proposition 3.2]{CM23}.
\begin{lemma}\label{sec}
Let $f$ be a diffeomorphism with a dominated splitting $TM=E\oplus_{\succ} F$. If
$\{\mu_n\}_{n\in \NN}$ is a sequence of different ergodic Gibbs $cu$-states of $E$ such that 
$\mu_n\xrightarrow[n\to +\infty]{weak^*} \mu$ and $\mu$ is a Gibbs $cu$-state of $E$, then for $\mu$-a.e. $x\in M$, there exist non-negative Lyapunov exponents along $F$.
\end{lemma}

We need the following result on pseudo-physical measures.

\begin{proposition}\label{ess}
Let $f$ be a $C^2$ diffeomorphism with a partially hyperbolic splitting 
$
TM=E^u\oplus_{\succ} E^c\oplus_{\succ} E^s.
$
If $E^c$ is $u$-definite, then any limit measure of pseudo-physical measures is a Gibbs $cu$-state of $E^u\oplus E^{cu}$.
\end{proposition}

\begin{proof}
Let $\{\mu_n\}$ be a sequence of pseudo-physical measures such that $\mu_n\to \mu$ for some $f$-invariant measure $\mu$ in the weak$^*$-topology. By Lemma \ref{ps1}, $\mu$ is also a pseudo-physical measure. By Remark \ref{hh}, $\mu_n, n\in \NN$, $\mu$ are Gibbs $u$-states. By Lemma \ref{sec}, for $\mu$-a.e. $x\in M$, there exist non-negative Lyapunov exponents along $E^c$. Since $E^c$ is $u$-definite, $\mu$ admits only positive Lyapunov exponents along $E^c$.
As $\mu$ is a pseudo-physical measure, by Lemma \ref{ps2} it satisfies 
$$
h_{\mu}(f)\ge \int \log |{\rm det}Df|_{E^u\oplus E^{c}}|d\mu.
$$
On the other hand, according to the sign of Lyapunov exponents of $\mu$, by Ruelle's inequality, it follows that
$$
h_{\mu}(f)\le \int \log |{\rm det}Df|_{E^u\oplus E^{c}}|d\mu.
$$
Hence, $\mu$ satisfies the Pesin entropy formula
$$
h_{\mu}(f) =\int \log |{\rm det}Df|_{E^u\oplus E^{c}}|d\mu,
$$
thus, $\mu$ is a Gibbs $cu$-state of $E^u\oplus E^{c}$. 
\end{proof}

We give the proof of Theorem \ref{fin} now.

\begin{proof}[Proof of Theorem \ref{fin}]
Observe firstly that all elements of $\mathcal{A}\cup \mathcal{B}$ are physical measures, so they are also pseudo-physical measures. 

Going by contradiction, we assume that $\mathcal{A}\cup \mathcal{B}$ is infinite. If $\mathcal{B}$ is infinite, then there exists a sequence of different Gibbs $cu$-states $\{\mu_n\}_{n\in \NN}$ in $\mathcal{B}$, which converges to some invariant measure $\mu$ in the weak$^*$-topology. By Proposition \ref{ess}, $\mu$ is also a Gibbs $cu$-state. We then derive the contradiction by applying Lemma \ref{sec}.

It remains to exclude the case when $\mathcal{A}$ is infinite. Argue by absurd, there exists a sequence $\{\mu_n\}_{n\in \NN}$ of different measures in $\mathcal{A}$ converging to some $f$-invariant measure $\mu$ in the weak$^*$-topology. By applying Proposition \ref{ess} again, $\mu$ is a Gibbs $cu$-state. Following the same argument in the proof of Theorem \ref{covering}, we can find an ergodic component $\nu$ of $\mu$ and an open subset $\mathcal{U}$ with the following properties: 
\begin{itemize}
\item $\nu$ is a Gibbs $cu$-state, so is in $\mathcal{B}$,
\smallskip
\item Lebesgue almost every point of $\mathcal{U}$ is contained in ${\rm Basin}(\nu)$,
\smallskip
\item $\mu(\mathcal{U})>0$.
\end{itemize}
Since $\mu_n$ converges to $\mu$ in the weak$^*$-topology as $n\to +\infty$, together with the third property above implies that
$$
\liminf_{n\to +\infty}\mu_n(\mathcal{U})\ge \mu(\mathcal{U})>0.
$$
As a result, $\mu_n(\mathcal{U})>0$ for all sufficiently large $n$. For such $n\in \NN$, as $\mu_n$ is an ergodic Gibbs $u$-state with negative central Lyapunov exponents, one can take a local strong unstable disk $\gamma$ inside $\mathcal{U}$ and a subset $\gamma_0\subset\gamma$ with constant $L_0$ so that 
\begin{itemize}
\item  $\gamma_0\subset {\rm Basin}(\mu_n)$ and ${\rm Leb}_{\gamma}(\gamma_0)>0$, 
\smallskip
\item all points of $\gamma_0$ admit Pesin local stable manifolds with uniform size $L_0$.
\end{itemize} 
Thus, by choosing $\rho<L_0$ small enough, the set 
$$
\mathcal{U}_{\rho}=\bigcup_{x\in \gamma_0} W^{cs}_{\rho}(x)
$$
is contained in $\mathcal{U}$, where $W^{cs}_{\rho}(x)$ denotes the local Pesin stable manifold of radius $\rho$ centered at $x$ tangent to $E^c\oplus E^{s}$. Observe that the Pesin stable manifold can only lie in the same basin, we have also that $\mathcal{U}_{\rho}\subset {\rm Basin}(\mu_n)$. To summarize, we have shown
\begin{equation}\label{c11}
\mathcal{U}_{\rho}\subset \mathcal{U}\cap {\rm Basin}(\mu_n).
\end{equation}
From the choice of $\gamma_0$ and the absolute continuity of Pesin stable lamination, it follows that 
$$
{\rm Leb}(\mathcal{U}_{\rho})>0.
$$
Together with (\ref{c11}), this implies  
$$
{\rm Leb}(\mathcal{U}\cap {\rm Basin}(\mu_n))>0.
$$
Recall that Lebesgue almost every point of $\mathcal{U}$ is contained in ${\rm Basin}(\nu)$, so we have
$$
{\rm Leb}({\rm Basin}(\nu)\cap {\rm Basin}(\mu_n))>0.
$$
Note that this holds for any sufficiently large $n$, so we get $\mu_n=\nu$ for all sufficiently large $n$, this would contradict to our assumption on $\{\mu_n\}_{n\in \NN}$.  The proof is complete. 
\end{proof}

\end{document}